\documentclass[reqno]{amsart}
\UseRawInputEncoding

\usepackage{amsmath,amsfonts,amssymb,amsthm,amscd,latexsym,cite}
\usepackage{mathrsfs}
\usepackage{color}
\newtheorem{thm}{Theorem}[section]
\newtheorem{theorem}[thm]{Theorem}

\newtheorem{lemma}[thm]{Lemma}

\newtheorem{corollary}[thm]{Corollary}
\newtheorem{proposition}[thm]{Proposition}
\newtheorem{definition}[thm]{Definition}

\newtheorem{rem}[thm]{Remark}
\newtheorem{question}{Question}
\newtheorem{example}[thm]{Example}
\newcommand{\cA}{{\mathcal A}}

\newcommand{\cE}{{\mathcal E}}

\newcommand{\cH}{{\mathcal H}}

\newcommand{\cM}{{\mathcal M}}

\newcommand{\cT}{{\mathcal T}}

 \usepackage{color}
  
 \newcommand{\norm}[1]{\left\lVert#1\right\rVert}
\usepackage{hyperref}					
\hypersetup{colorlinks,
	linkcolor=blue,%
	citecolor=blue}

 \usepackage{float}
\begin{document}
\title[Norms of multiplication operators]{Norms of multiplication operators:  answering Fialkow--Loebl question}

\author[Jinghao Huang]{J. Huang}\thanks{Y. Zhu is the corresponding author. J. Huang was supported by the NNSF of China  (No. 12031004, 12301160 and  12471134). F. Sukochev was supported by the ARC (DP230100434)}
\address{ Institute for  Advanced Study in  Mathematics of HIT, Harbin Institute of Technology, Harbin, 150001, China}
\email{{\color{blue}jinghao.huang@hit.edu.cn}}
\email{{\color{blue}xxurann@stu.hit.edu.cn}}
\email{{\color{blue}24s012030@stu.hit.edu.cn}}

\author[Fedor A. Sukochev]{F. Sukochev} 
\address{School of Mathematics and Statistics, University of New South Wales, Sydney, Australia}
\email{{\color{blue}f.sukochev@unsw.edu.au}}

\author[Ran Xu]{R. Xu}

\author[Yunpeng Zhu]{Y. Zhu} 

\subjclass[2020]{46L10; 46L52; 47B47. \hfill 
}


\begin{abstract} 

Let $\cM$ be a
factor equipped with a semi-finite  faithful normal trace $\tau$.
Let $E(0,\infty)$ be a symmetrically normed
function space and $E(\cM,\tau)$ be the corresponding symmetrically normed
operator space.
Suppose that  $a, b$ are  $\tau$-measurable operators affiliated with $\cM$.
It is shown  that  the range of  the multiplication operator  $S_{a,b}: x\mapsto axb$  on $\cM$ is contained in 
 $E(\cM, \tau)$ if and only if       
 $\mu(a)\mu(b)$ belongs to $E(0, \infty)$, where $\mu(x)$ stands for the generalized singular value function of a $\tau$-measurable operators $x$ affiliated with $\cM$.
 Moreover, we have 
 \begin{align*}
 	\norm{S_{a,b}}_{\cM\to E(\cM,\tau)}=
 	\norm{   \mu(a )\mu(  b) }_{E (0,\infty)   },
 \end{align*}
 which answers a question by Fialkow and Loebl (1984).
We also consider the quasi-normed case, and show that the natural quasi-norm of weak $L_p$-space, $0<p<\infty$, is not monotone with respect to the logarithmic submajorisation.
\end{abstract}
\maketitle

\section{Introduction}
\subsection{Background of multiplication operators}
Let $B(\cH)$ be the $*$-algebra of all bounded linear operator on a Hilbert space $\cH$.  
The {\it multiplication operator} $ S_{a, b}$, for  $a,b \in B(\cH)$,  is defined by $$S_{a, b}(x)=axb,$$
 for all $x\in B(\cH)$ \cite{Harte}.
 The range and norm of multiplication operators (and its generalizations, the so-called elementary operators) have been subjects of extensive research in operator theory (see e.g. \cite{F85,ST,F84,CM,EFHMMS} and references therein).
Fialkow and Loebl  proved in \cite[Theorem 5.6]{F84} that,  the range of $S_{a, b}$ is contained in a proper two-sided ideal $ \mathscr{J}$ of $B(\cH)$ if and only if $\mu(a)\mu(b) \in J$,
where $\mu(\cdot)$ stands the sequence $\{\mu_n(\cdot)\}_{n = 1}^{\infty}$ of singular values of an operator, 
and $J$ is the symmetric sequence space corresponding to  $\mathscr{J}$\cite{Kalton_S,LSZ}. Moreover, if the range of $S_{a, b}$ is contained in the Schatten $p$-class $C_p$, $1\le p<\infty $,
then we have 
\begin{align}\label{C_p}
	\norm{S
 	_{a, b}}_{B(\mathcal{H}) \to C_p}=
 	\norm{\mu(a)\mu(b)}_{\ell_p},
\end{align}
see \cite[Theorem 5.7]{F84}. 

 Fialkow and Loebl posed the following question.
\begin{question}\label{q:FL}
	\textit{We have not as yet determined whether Theorem~5.7 can be extended to all of the $C_\Phi$ ideals. More generally, the calculation of the norm of $S_{a,b}:B(\cH)\to  \mathscr{J}$ for an arbitrary norm ideal $\mathscr{J}$ is an open problem \cite[p.~572]{F84}.
    $\cdots \cdots$
    Concerning the operator $S_{a,b}$, is the identity $\left\|S_{a,b}\right\| = \Phi(\mu(a)\mu(b))$ valid \cite[p.~577]{F84}}? 
\end{question}
Here,   the $C_\Phi$ are operator  ideals  in $B(\cH)$  defined in \cite{Gohberg}.
In \cite[Theorem~5.6]{F84}, Fialkow and Loebl established the following estimate for the norm of $S_{a,b}$:
\begin{align}\label{5.6}
\Phi(\mu(a)\mu(b)) \leq \left\|S_{a,b}\right\|_{B(\cH)\to C_\Phi} \leq 2  \Phi(\mu(a)\mu(b)).
\end{align}
 The main purpose of the present paper is to answer Fialkow--Loebl question and its semifinite  counterpart. 
\subsection{Multiplication operators on   symmetrically normed operator spaces}
Let $\cM$ be a factor on a Hilbert space ${\cH}$,  equipped with a semi-finite  faithful normal trace $\tau$
and let $E(0,\infty)$  be a symmetrically normed   function space on $(0,\infty)$  (see Definition~\ref{Sym} below or \cite{DPS}).
The associated symmetrically normed
operator space  $E(\cM,\tau)$ is defined as a subset of the $\ast$-algebra $S(\cM,\tau)$ of $\tau$-measurable operators affiliated with $\cM$ by
\[
x\in E(\cM,\tau)\quad\Leftrightarrow\quad \mu(x)\in E(0,\infty), 
\]
where $\mu( x)$ is the  singular value function of $x$ \cite{Kalton_S, LSZ}. 
The associated norm of $E(\cM,\tau)$ is given by \cite{Kalton_S, LSZ}
\[
\left\|x\right\|_{E(\cM,\tau)} = \|\mu(x)\|_{E (0,\infty)   }.
\]

For $a,b\in S(\cM, \tau)$, let $S_{a,b}$  denote the multiplication operator 
 on $ \cM $ defined by 
$$S_{a,b}(x)=axb.$$
Our first result, Theorem~\ref{main} below, provides a characterization of the inclusion $\mathrm{Ran}(S_{a,b}) \subseteq E(\cM,\tau)$ and an explicit formula for the norm of $S_{a,b}$ with respect to a symmetric space.
In the proof of \cite[Theorem 5.7]{F84}, 
the inequality 
\begin{align}\label{sjab}
\sum_{j = 1}^{\infty}{\mu_j(ab)}^p \leq \sum_{j = 1}^{\infty}{\mu_j(a)}^p{\mu_j(b)}^p, ~
 a, b \in B(\cH),~p \geq 1, 
 \end{align}
  plays a key role in establishing 
$\norm{S
 	_{a, b}}_{B(\mathcal{H}) \to C_p} \leq
 	\norm{\mu(a)\mu(b)}_{\ell_p}.$ 
However, inequality  \eqref{sjab} is not applicable in the setting of   general symmetric spaces. In \cite[Theorem 7]{Sukochev16}, a stronger result was established:
$$ab\vartriangleleft\mu(a)\mu(b), ~\forall a, b\in S(\cM,\tau) , $$
where $\vartriangleleft$  stands for the {\it uniform submajorisation} introduced in \cite{Kalton_S} (and thoroughly discussed in \cite[Chapter 3]{LSZ}).
This allows us to consider a more general class of symmetrically normed operator spaces (in particular, ideals in $B(\cH)$) than those considered by  Fialkow and Loebl in \cite{F84} without the restrictive conditions imposed on such ideals in the books of Schatten~\cite{Schatten}, Gohberg and Krein\cite{Gohberg} and Simon~\cite{Simon79,Simon05}.

 \begin{theorem}\label{main}
 	Suppose that
 	$\cM$ is a factor  equipped with a semi-finite faithful normal trace $\tau$.
	Let $E(0,\infty)$ be a symmetrically normed function space. 
    For any  $a,b \in S(\cM,\tau)$,
 	 the range of $S_{a,b}$ is contained in 
 	$E(\cM, \tau)$ if and only if 
 	$\mu(a)\mu(b)\in E(0, \infty)$.
 	Moreover,  we have
 	\begin{align*}
 		\norm{S_{a,b}}_{\cM\to E(\cM,\tau)}=
 		\norm{   \mu(a )\mu(  b) }_{E (0,\infty)   }.
 	\end{align*}
 \end{theorem}
In the special setting when $\cM = B(\cH)$ (and ideals are understood in \cite{F84,Gohberg,Schatten}), Theorem \ref{main} answers Question \ref{q:FL} and  improves  inequality \eqref{5.6}. Note that  operators $a$ and $b$ are not required to be compact as in \cite{F84}, and $E(\cM,\tau)$ is not necessarily complete. 
\subsection{Multiplication operators on symmetrically quasi-normed operator spaces}
There are numerous examples of spaces which are only quasi-normed but not normed, e.g., the weak $L_p$-spaces, $0< p<\infty  $,  and $L_p$-spaces, $0<p<1$. 
It is natural to consider Fialkow and Loebl's question, Question \ref{q:FL} above, in the setting of symmetrically quasi-normed ideals in $B(\cH)$, and ask whether the formula in Theorem~\ref{main} still holds. 
Surprisingly, it turns out that  this is not the case. 
Consider   the natural quasi-norm  of the weak $L_p$-space $L_{p,\infty}(0,\infty ) $ ($p>0$) given by 
$$
\norm{f}_{p,\infty}=\sup_{t>0} t^{1/p} \mu(t;f),~f~\mbox{is~a measurable function on~} (0,\infty).
$$
The recent resolution \cite{SZ21} of a problem  due to Simon\cite{Simon79,Simon05} concerning the optimal constants in the H\"older inequalities for $\norm{\cdot}_{p,\infty }$  yields that 
$\norm{\cdot}_{p,\infty}$ does not satisfy the formula in Theorem \ref{main} (see Example~\ref{Lp} below).

On the other hand, we  consider Fialkow--Loebl question in the setting of  a wide class of symmetrically quasi-normed spaces, i.e., 
 symmetrically quasi-normed operator spaces  whose quasi-norms are monotone with respect to the logarithmic submajorisation (denoted by $\prec\prec_{\log}$, see e.g. \cite{SZ18,Hiai,DDSZ} or Section \ref{s2} for the definition), and obtain the following result.
\begin{theorem}\label{main2}
		Suppose that
		$\cM$ is a   factor equipped with a semi-finite  faithful normal  trace~$\tau$. Let $E(0,\infty)$ be a symmetrically quasi-normed function space whose quasi-norm is monotone with respect to $\prec\prec_{\log}$.  
		For any $a,b \in S(\cM,\tau)$,
		the range of $S_{a,b}$ is contained in 
 	    $E(\cM, \tau)$ if and only if 
 	    $\mu(a)\mu(b)\in E(0,\infty)$.
		Moreover, 
			\begin{align}\label{mainequal2}
			\norm{S_{a,b}}_{\cM\to E(\cM, \tau)}=
			\norm{  \mu(a)\mu(b) }_{E(0,\infty)  }.
		\end{align}		
	\end{theorem}
Most of  known examples of  symmetrically (quasi)-normed
function spaces $E(0, \infty)$ such as Lorentz space $L_{p,q}(0,\infty)$, $ 0<p, q<\infty $ (see Example~\ref{Lorentz} below)  and the weighted Lorentz space $\Lambda ^p_{\omega }$, $ 0<p<\infty $ (see \cite[Lemma 6.4]{Huang2}), have (quasi)-norms that are monotone with respect to $\prec\prec_{\log}$. For any $0<p<\infty$, the (quasi)-norm of $L_p(\cM,\tau)$  is also monotone with respect to $\prec\prec_{\log}$ (see e.g. \cite[Lemma 6.3]{Huang2}). Hence, Theorem \ref{main2} extends \cite[Theorem 5.7]{F84} (see equality \eqref{C_p} above). Moreover, we establish that every symmetrically quasi-normed
function space has an equivalent quasi-norm  that is monotone with respect to $\prec\prec_{\log}$ (see Theorem~\ref{log} below), which extends \cite[Proposition 2]{Fack}. 
Recall $\norm{\cdot}_{p,\infty}$ has an equivalent fully symmetric norm (in particular, monotone with respect to $\prec\prec_{\log}$) when $p>1$\cite{Bennett_S} (see also \cite[Proposition 6.7.8]{DPS}). 
However, since $\norm{\cdot}_{p,\infty}$ does not satisfy \eqref{mainequal2}, it follows that the natural quasi-norms of $L_{p,\infty}(0,\infty)$ ($p>0$) fail the monotonicity with respect to $\prec\prec_{\log}$  (see Example~\ref{Lp}), which is unexpected.

\section{Preliminaries}
\label{s2}

In this section,
we recall  notions which are  needed
in this paper. 
General information on von Neumann
algebras and symmetric spaces  can be found in
\cite{LSZ,KPS,LT2,Sukochev17}.
\subsection{$\tau$-measurable operators}

Let \((I, m)\) denote the measure space  \( I = (0, a) \) for some \( a \in (0, \infty] \), endowed with the Lebesgue measure \( m \). Denote by \( L(I, m) \) the space of all measurable functions on \( I \), where functions that agree almost everywhere are identified. Let \( S(I, m) \) (or simply \( S(I) \)) be the linear subspace of \( L(I, m) \) consisting of all functions \( x \) for which there exists some \( s > 0 \) such that
\[
m(\{ t \in I : |x(t)| > s \}) < \infty.
\]

For $x\in S(I)$, $\mu(x)$ denotes the decreasing rearrangement of the function $|x|$~\cite{KPS,LSZ,LT2},   that is,
$$
\mu(t; x)=\inf \left\{s\geq0:\ m(\{|x|>s\})\leq t \right\},\quad t>0.
$$

Let $\cM$ be a semi-finite von Neumann algebra on a  Hilbert space $\cH$,  with a semi-finite  faithful normal trace $\tau$, and 
identity ${\bf 1}$.
Let $P(\mathcal{M})$ denote the lattice of all projections in $\mathcal{M}$, and let
$U(\cM)$ denote the set of all unitary elements in $\cM$.
A closed and densely defined operator $x$, affiliated with $\mathcal{M}$, is called \emph{$\tau$-measurable} if
$\tau(e ^{|x|}(s,\infty))<\infty$ for sufficiently large $s$, where $e ^{|x|}$ denotes  the spectral measure of $|x|$.
The set of all $\tau$-measurable operators is denoted by
$S(\mathcal{M},\tau)$\cite{LSZ,DPS}.
For every $x\in S(\mathcal{M},\tau),$ the {\it singular value function} $\mu(x)$ is defined by setting \cite{DPS,LSZ} 
$$\mu(t; x)=\inf\{s\geq0:\ \tau(e ^{|x|}(s,\infty))\leq t \},\quad t>0.$$

\subsection{Symmetric spaces}

 For the convenience of the reader, we recall that a map $\norm{\cdot}$ from a linear space $X$ into the field $\mathbb{R}$ of real numbers is a \emph{quasi-norm} if for all $x, y \in X$ and scalars $\alpha$, the following properties hold:

\begin{enumerate}
    \item[(i)] $\|x\| \geq 0$, and $\|x\| = 0$ if and only if $x = 0$;
    \item[(ii)] $\|\alpha x\| = |\alpha| \|x\|$;
    \item[(iii)] $\|x + y\| \leq C (\|x\| + \|y\|)$, for some constant $C \geq 1$.
\end{enumerate}
The couple $(X, \norm{\cdot})$ is called a \emph{quasi-normed space}, and the least constant $C$ satisfying the inequality (iii) above is called the \emph{modulus of concavity} of the quasi-norm $\norm{\cdot}$ and is denoted by $C_{X}$.

\begin{definition}\cite{Sukochev17,DPS}
\label{Sym}
We say that $(E(I),\left\|\cdot\right\|_E)$ (or    $(E,\left\|\cdot\right\|_E)$ for brevity) is a symmetrically normed (respectively, quasi-normed)  function space
on $I$ if the following hold:
\begin{enumerate}
\item $E$ is a subspace of $S(I)$;
\item $(E,\left\|\cdot\right\|_E)$ is a normed (respectively, quasi-normed)  space;
\item If $x\in E$ and $y\in S(I)$  are such that $\mu(y)\leq\mu(x),$ then $y\in E$ and $\left\|y\right\|_E\leq \left\|x\right\|_E.$
\end{enumerate} 
\end{definition}

For a symmetrically normed  function space $ E $ and $ x\in E$, if $ y \in S(I)$ is such that $y\vartriangleleft x$, then $y \in E$ and $\norm{y}_E\leq \norm{x}_E$ \cite[Corollary 3.4.3]{LSZ}, where $\vartriangleleft$ stands for the {\it uniform submajorisation} introduced in \cite{Kalton_S}, meaning that there  exists $\lambda \in \mathbb{N} $ such that
\begin{align*}
	\int_{\lambda a}^{b} \mu(s; x) \, ds \leq \int_{a}^{b} \mu(s; y) \, ds, ~ \lambda a \le b. 
\end{align*}

For the general theory of symmetrically normed (respectively, quasi-normed)  function spaces, we refer the reader to \cite{LT2,DPS,KPS}.

\begin{definition}\cite{DPS,Sukochev17}
Let $  \cE$   be a linear   subspace of    $S(\mathcal{M},\tau)$, equipped with a 
norm (respectively, quasi-norm) $\left\|\cdot\right\|_{\cE}$.
We say that
$ \cE $ is a symmetrically  normed (respectively, quasi-normed)
operator space if for $ x\in \cE $  and for every $y\in S(\mathcal{M},\tau)$ with
$\mu(y)\leq\mu(x),$ we have $y\in \cE $
and $\left\|y\right\|_{\cE} \leq
\left\|x\right\|_{ \cE} $.

\end{definition}

Let
$E$ be a symmetrically normed (respectively, quasi-normed)  function   space on $(0,\tau({\bf 1}))$.
Set
$$
E(\mathcal{M},\tau)=\Big\{x\in S(\mathcal{M},\tau):\ \mu(x)\in E\Big\}
,$$
and
\begin{align*}
\left\|x\right\|_{E (\mathcal{M},\tau)}:= \left\|\mu(x)\right\|_E,\quad x \in E(\mathcal{M},\tau).
\end{align*}
The space $E(\mathcal{M},\tau)$ is called the symmetrically normed (respectively, quasi-normed) operator space corresponding to $(E,\left\|\cdot\right\|_{E})$\cite{Kalton_S,Sukochev17}.

 A symmetrically normed (respectively, quasi-normed) function space $E(0,\infty)$ is called a Banach (respectively, quasi-Banach) symmetric function space if it is complete. 
 If $E(0,\infty)$ is complete, then $E(\cM,\tau)$ is also complete\cite{Kalton_S,Sukochev17}.

When $E=L_p(0, \infty)$ for  $1 \leq p \leq \infty$, the associated  space $E(\cM, \tau) $ coincides with the noncommutative $L_p$-space  $L_p(\cM, \tau)$. In particular, $L_\infty(\cM, \tau) = \mathcal{M}$. 
When  $\mathcal{M}$ is $B(\mathcal{H})$, $L_p(\cM, \tau)$  is the   Schatten-$p$ class \cite{LSZ}.

\subsection{The algebra \( L_{\log_+}(\cM, \tau) \) }
Let $\cM$ denote a semi-finite von Neumann algebra, equipped with a  semi-finite  faithful normal trace~$\tau$.
 Set
\[
L_{\log_+}(\cM,\tau) := \{x \in S(\cM, \tau) : \log_+ |x| \in L_1(\cM,\tau)+L_\infty(\cM,\tau)\},
\]
where 
\[
\log_+ t := \max\{\log t, 0\}.
\]
The space $L_{\log_+}(\cM,\tau)$ forms a $*$-algebra \cite{Dykema} with the following inclusions by \cite[p.~ 7]{DDSZ}:
\[
L_1(\cM,\tau)+L_\infty(\cM,\tau)
\subseteq L_{\log_+}(\cM,\tau) 
\subseteq S(\mathcal{M},\tau) .
\]
 If \( E(0,\infty ) \subseteq S(0,\infty ) \) is a symmetrically quasi-normed function space,  
then there exists $q\in (0,\infty)$ such that
$$E(\cM,\tau) \subseteq L_q(\cM,\tau)+L_\infty(\cM,\tau) \subseteq L_{\log_+}(\cM,\tau)$$
(see \cite[Remark 2]{Schechtman} and \cite[Lemma 3.4.1]{Zanin}).

If \( x, y \in S(\cM, \tau) \), then we say that \( y \) is \emph{logarithmically submajorised} by \( x \), if and only if 
\[
\int_0^t \log \mu(s; y) \, ds \leq \int_0^t \log \mu(s; x) \, ds,  \quad \forall t > 0,
\]
denoted by \( y \prec\prec_{\log} x \), see e.g.  \cite[p.~125]{Hiai} and  \cite{DDSZ,Huang2}.

The quasi-norm $\left\|\cdot\right\|_{E(\cM, \tau)}$
on the symmetrically quasi-normed operator space \( E(\cM, \tau) \) 
is said to be \emph{monotone with respect to the logarithmic submajorisation} if whenever \( y \in E(\cM, \tau) \), \( x \in L_{\log_+}(\cM, \tau) \) satisfy \( x \prec\prec_{\log} y \), it follows that \( x \in E(\cM, \tau) \)  and
\( \left\|x\right\|_{E(\cM, \tau)} \leq \left\|y\right\|_{E(\cM, \tau)} \).

\section{Proof of Theorems \ref{main} and \ref{main2}}

For any measurable function $x$ on $(0,\infty)$, the dilation $D _s(x)$, $s>0$,  is defined by\cite[p.~392]{DPS} as:
 $$(D_s(x))(t)=x\left(\frac{t}{s} \right ),\quad t\in(0,\infty).$$

	The following result extends \cite[Lemma 5.5]{F84}. 
	\begin{lemma}\label{E0EM}
		Let $\cM$ be a semi-finite von Neumann algebra,   equipped with a semi-finite  faithful normal trace~$\tau$.  Suppose that  $a, b\in S(\cM,\tau)$.   
		If 
		$\mu(a)\mu(b)$ belongs to the  symmetrically quasi-normed    function space $E(0, \infty)$, then 
		 the range $\mathrm{Ran}(S_{a,b})$  of $S_{a,b}$  is contained in $E(\cM, \tau)$.
	\end{lemma}
	\begin{proof}
			To show that $\mathrm{Ran}(S_{a,b}) \subseteq  E(\cM, \tau)$, it suffices to prove that 
			 $\mu(axb) \in E(0, \infty)$ for all $x \in \cM $.
		Since \(a, b \in S(\cM, \tau)\) and $x\in\cM$,  it follows from \cite[Proposition 3.2.7(iv), (vi)]{DPS} that
	\begin{align*}
		\mu(s; a x b) &
		\leq \mu(s/2; ax)  \mu(s/2; b) \\ &\nonumber
		\leq \norm{x}_{_{\cM}}
		\mu(s/2; a)  \mu(s/2; b) \\ &\nonumber
		= \norm{x}_{_{\cM}}
			\Big(
		D_2
	(
		 \mu(a)
		 \mu(b)
		)
		 	 \Big)
		 (s), 
		 \quad s > 0.
	\end{align*}		
		By assumption that $\mu(a)\mu(b)\in E(0, \infty)$ and by
		  \cite[Lemma 6.1.1]{DPS},    \(
		D_2( \mu(a)\mu(b))
		\)  also belongs to \(E(0, \infty)\).
		By the definition of symmetrically  quasi-normed    function spaces, we have 
	\(\mu(axb) \in E(0, \infty)\), that is, \(axb\in E(\cM, \tau)\).		
	\end{proof}

\begin{rem}
    The converse of Lemma \ref{E0EM} fails even for symmetrically quasi-normed function spaces.
    For example, let $a,b\in E(0,\infty)$ be such that $ab=0$. In particular, we have $S_{a,b}=0$, which means that the range $ \mathrm{Ran}(S_{a,b}) \subseteq  E(0,\infty )$. However, $\mu(a)\mu(b)$ is not necessarily in $E(0,\infty )$. 
\end{rem}

The following lemma is certainly known to experts. For a complete proof, we refer to \cite{HUANG}.

\begin{lemma}\label{lemma:decomposition}\cite[Lemma 4.1]{HUANG}
Let $\cM$ be an atomless von Neumann algebra equipped with a semi-finite faithful normal trace $\tau$.
Suppose that $x\in S(\cM,\tau)$  such that $\mu(x)$ is continuous on $(0,\tau(s|x|))$  and let 
 $N \ge 2$ be a natural number. 
There exists  $0 <\cdots<s_i<s_{i+1} < \cdots < \tau(s(|x|))$, 
$-\infty <i<\infty $, 
such that 
 $$ \mu(s_i;x )  \ge \mu(s_{i+1};x )> \frac{N-1}{N}  \mu(s_i;x) . $$
\end{lemma}

The following two lemmas are key ingredients in the proof of Theorem \ref{main}. 
The main tool is the total comparability of projections in a factor\cite[Proposition 6.2.6]{KR-II}.

For $x\in S(\cM,\tau)$,
the  non-negative number $\mu_\infty^x$ is defined by \cite[p.~190]{DPS}
$$\mu_\infty^x=\lim\limits_{t\to\infty}\mu(t;x)$$ 
and $t^x$ is defined by 
$$t^x:=\inf\{t>0: \mu(t;x)=\mu_\infty^x\}.$$

\begin{lemma}\label{compact}
	Let $\cM$ be a factor equipped with a semi-finite  faithful normal trace $\tau$, and let $E(0,\infty )$ be a symmetrically quasi-normed function space.
	Assume that  $0 \leq a,b \in \{x \in S(\cM,\tau):\mu^x_{\infty} = 0 \text{ or } t^x = \infty \}$ and  $\mu(axb)\in E(0, \infty)$ for all $x\in\cM$.
    We have
	$\mu(a)\mu(b)\in E(0, \infty)$,  and for  arbitrary $\varepsilon > 0$,  there exists a partial isometry  $w_{_{\varepsilon}} \in \cM$
	 such that  
	\begin{align*}
		(1-\varepsilon)
	\mu(a)\mu(b)
\leq
	\mu	(
 aw_{_{\varepsilon}}
	b
)
\leq
		(1+\varepsilon)
	\mu(a)\mu(b)
	.
	\end{align*}  
Consequently, 
	\begin{align*}
	(1-\varepsilon)
	\norm{\mu(a)\mu(b)	}_{E(0,\infty)}
	\leq
	\norm{ 
	 aw_{\varepsilon}
		b
	}_{E(\cM, \tau )}
		\leq
	(1+\varepsilon)
	\norm{ 	\mu(a)\mu(b)
	}_{E(0,\infty)}
	.
	\end{align*}	
\end{lemma}
\begin{proof}
	    If $x\in S(\cM, \tau)$ with $\mu^x_{\infty} = 0$, then \begin{align}\label{mua2}
		xe^x(\mu^x_{\infty}, \infty)=	xe^x(0, \infty)=x.
	    \end{align}	
	    If $x\in S(\cM, \tau)$ with $t^x=\infty$, 
		then by the definition of $t^x$, we have 
		$\tau\left(e^x (\mu_\infty^x ,\infty )\right) =t^x=\infty$ (see e.g.\cite[Eq.(3.46)]{DPS}).
		It follows from \cite[Proposition 3.2.10 (iii)]{DPS}
		that 
		\begin{align}\label{mua1}
			\mu(xe^x(\mu^x_{\infty}, \infty))=\mu(x)\chi_{[0, \infty)}=\mu(x).
		\end{align} 
        
         We only consider the case when  $\cM$ is atomless.  The case for atomic factors can be
         proved similarly (see also  the proofs of  \cite[Lemmas 5.3 and 5.4]{F84}).  
       Without loss of generality, we assume that $0<\tau(s(a))\leq \tau(s(b))$.
       Firstly, we consider the case when  $\mu(a),\mu(b)$ are continuous on $(0,\infty )$.

	   {\bf Step 1}.
We first consider  operators 
          $ae^a(\mu_\infty^a,\infty )$ and $be^b(\mu^b_{\infty}, \infty)$.
 For  simplicity, we denote 	$$0\leq\widetilde{a}:=ae^a(\mu_\infty^a,\infty )=e^a(\mu_\infty^a,\infty )ae^a(\mu_\infty^a,\infty ),$$ $$0\leq\widetilde{b}:=be^b(\mu_\infty^b,\infty )=e^b(\mu_\infty^b,\infty )be^b(\mu_\infty^b,\infty ).$$ 
	   Let $N \ge 2$ be a fixed natural number.	By Lemma \ref{lemma:decomposition}, we obtain  $0 <\cdots<s_j<s_{j+1} < \cdots < \tau(s(b))$, $-\infty <j<\infty $, such that 
	   \[
	   \begin{cases}
	   \mu(s_j;a)   \ge \mu(s_{j+1};a)> \frac{N-1}{N}  \mu(s_j;a),& s_j<\tau(s(a)), \\
	   \mu(s_j;a)=0, & s_j\geq \tau(s(a))	,
	   \end{cases}
	   \]
	   and
	   \[
	   \mu(s_j;b)  \ge \mu(s_{j+1};b)> \frac{N-1}{N}  \mu(s_j;b) .
	   \]
	   Let $t_j:=\mu(s_j;a)$  and
	   $t'_j:=\mu(s_j;b)$.  
	   Then, we have
	   \begin{align}\label{lambdaj}
	\nonumber	t_j \geq t_{j+1} > \frac{N-1}{N}t_j, \quad 0<s_j<\tau(s(a)),&\\ 
        t'_j \geq t'_{j+1} > \frac{N-1}{N}t'_j, \quad  0<s_j<\tau(s(b)).
	   \end{align}

	Since $\cM$ is atomless,
	it follows from
	\cite[Lemma 3.7.7(i)]{DPS} (or \cite[Lemma 3.7.8]{DPS})
	that for  $\{t_j\}_{-\infty<j<\infty}$ and
	$\{t'_j\}_{-\infty<j<\infty}$, 
	there exist two sequences $\left\{e_j \right\}_{-\infty<j<\infty}$,
	$\left\{e'_j \right\}_{-\infty<j<\infty}$ of projections in 
	$\cM$ such that 
	$$  \cdots\leq e_j\leq e_{j+1}  \cdots \mbox{ and }
	\tau(e_j) =s_j, ~\forall -\infty<j<\infty,
	$$
	$$  \cdots\leq e'_j\leq e'_{j+1}  \cdots \mbox{ and }
	\tau(e'_j) =s_j, ~\forall -\infty<j<\infty,
	$$
	$$ e^{a}(t_j,\infty)\leq e_j \leq 
	e^{a}[t_j,\infty )\quad \mbox{
		and  }\quad
	e^{b}(t'_j,\infty)\le e_j' \le 
	e^{b}[t'_j,\infty ).$$

Since $\cM$ is a factor, 
it follows from \cite[Corollary 6.2.6]{KR-II} that  there exists a  partial isometry  $w_{_{N}} \in \cM $  such that  
\begin{align}\label{equal1}
	w_{_{N}} e'_j w_{_{N}} ^*=e_j, ~\forall -\infty<j<\infty. \end{align}
	Observe that,  we have
	\begin{align} \label{ineqa}
		\begin{split}
			\sum_{-\infty <j<\infty } t_{j+1}  (e_{j+1}-e_j) \stackrel{\eqref{lambdaj}}{\leq} \widetilde{a} 
			&  \stackrel{\eqref{lambdaj}}\leq   \sum_{-\infty <j<\infty  } t_{j }  (e_{j+1}-e_j) 
			\\ &
			\stackrel{\eqref{lambdaj}}{
			<}\frac{N}{N-1} \Big( \sum_{-\infty <j<\infty  } t_{j+1 }  (e_{j+1}-e_j) \Big )
		\end{split}
	\end{align} 
	and 
	\begin{equation}\label{be}
		\begin{split}
		\sum_{-\infty <j<\infty } t'_{j+1}  (e'_{j+1}-e'_j) \stackrel{\eqref{lambdaj}}{\leq} 
	    \widetilde{b}
		&\stackrel{\eqref{lambdaj}}{\leq}\sum_{-\infty <j<\infty } t'_{j}  (e'_{j+1}-e'_j)
		\\ &
		 \stackrel{\eqref{lambdaj}}{<}\frac{N}{N-1} \Big(
		 \sum_{-\infty <j<\infty  } t'_{j+1 }  (e'_{j+1}-e'_j)
		  \Big ).
		\end{split}
	\end{equation}
By \cite[Example 3.2.2]{DPS}, we obtain that 	
	\begin{align} \label{1.3}
	\sum_{-\infty <j<\infty  } t_{j+1 }\chi_{[ s_{j}, s_{j+1} )} 
	\stackrel{\eqref{ineqa}}{\leq}\mu\left(\widetilde{a}
	\right) 
	\stackrel{\eqref{ineqa}}{<}\frac{N}{N-1} \sum_{-\infty <j<\infty  } t_{j+1 }\chi_{[ s_{j}, s_{j+1} )},
    \end{align}
    and 
    \begin{align} \label{1.4}
		\sum_{-\infty <j<\infty  } t'_{j+1 }\chi_{[  s_{j}, s_{j+1} )}
		\stackrel{\eqref{be}}{\leq}	\mu(\widetilde{b})
		\stackrel{\eqref{be}}{<}\frac{N}{N-1} \sum_{-\infty <j<\infty } t'_{j+1 }\chi_{[  s_{j}, s_{j+1} )}.
    \end{align}
Hence,	 we have
\begin{align} \label{1.5}
	\begin{split}		
		\sum_{-\infty <j<\infty  }t_{j+1 } t'_{j+1 }\chi_{[  s_{j}, s_{j+1} )}
		& \stackrel{\eqref{1.3},\eqref{1.4}}{
			\leq}
		\mu(\widetilde{a})\mu(\widetilde{b})\stackrel{\eqref{mua2},\eqref{mua1}}{=}\mu(a)\mu(b)
		\\	&
		\stackrel{\eqref{1.3}, \eqref{1.4}}{<}
		\left(\frac{N}{N-1}\right)^2 \sum_{-\infty <j<\infty  }t_{j+1 } t'_{j+1 }\chi_{[  s_{j}, s_{j+1})}.
	\end{split}
\end{align}

		{\bf Step 2}. 
Observe that, we have
\begin{eqnarray}\label{mu(u*au)}	
		 	\sum_{-\infty <j<\infty  } t_{j+1}  (e'_{j+1}-e'_j) &	\stackrel{\eqref{equal1}}{=} &	w_{_{N}} ^*   \Big(\sum_{-\infty <j<\infty } t_{j+1}  (e_{j+1}-e_j)\Big) w_{_{N}} \nonumber\\ 
        &  \stackrel{\eqref{ineqa}}{\le} & 
        w_{_{N}} ^*   
	\widetilde{a}		
			w_{_{N}}   	\\		&\stackrel{\eqref{equal1},\eqref{ineqa}}{<}&
            \frac{N}{N-1} 
			\Big( \sum_{-\infty <j<\infty  } t_{j+1 }  (e'_{j+1}-e'_j) \Big ).  	\nonumber
\end{eqnarray}	
By \cite[Lemma 3.7.10]{DPS}, we have
\[
	   \begin{cases}
	   ae_j=e_ja,& s_j<\tau(s(a)), \\
	   ae_j=as(a)=e_ja, & s_j\geq \tau(s(a)),	
	   \end{cases}
	   \]
and $be'_{j}=e'_{j}b$ for all $j$, which together with the equality \eqref{equal1} yield that
\begin{align*}
	ae_j=e_j a \stackrel{\eqref{equal1}}{\Longleftrightarrow } aw_{_{N}} e'_j w_{_{N}} ^*= w_{_{N}} e'_j w_{_{N}} ^* a \Longleftrightarrow w_{_{N}}^*aw_{_{N}}e_{j}'=e_{j}'w_{_{N}}^*aw_{_{N}}. 
\end{align*}
By \cite[Proposition 2.2.24 (iii), (iv)]{DPS}, we have 
\begin{eqnarray*}
& &		\sum_{-\infty <j<\infty } (t'_{j+1})^{2}  (t_{j+1})^{2} 
	(e'_{j+1}-e'_j) \\
	& \stackrel{\eqref{mu(u*au)}}{	\leq} &
    \sum_{-\infty <j<\infty } (t'_{j+1})^{2} \Big(w_{_{N}} ^*   \widetilde{a}
	w_{_{N}} (e'_{j+1}-e'_j) \Big)^{2}
	\\ &  =&  w_{_{N}} ^*   \widetilde{a}
	w_{_{N}}
    \left( 
    \sum_{-\infty <j<\infty } (t'_{j+1})^{2} (e'_{j+1}-e'_j)  
    \right)
    w_{_{N}} ^*   \widetilde{a}w_{_{N}}
	\\ &
	\stackrel{\eqref{be}}{	\leq}&
	w_{_{N}} ^*   \widetilde{a}
	w_{_{N}} 	\widetilde{b}^2
	w_{_{N}} ^*   \widetilde{a}
		 w_{_{N}} 
	\\ &\stackrel{\eqref{mu(u*au)}, \eqref{be}}{	<}&
	\left(\frac{N}{N-1}\right)^4	\sum_{-\infty <j<\infty } (t'_{j+1})^{2} (t_{j+1})^{2} 
	(e'_{j+1}-e'_j) .
\end{eqnarray*}
By the above inequality  and \cite[Example 3.2.2]{DPS}, we have
\begin{eqnarray*}
& &
\sum_{-\infty <j<\infty } (t'_{j+1})^{2} (t_{j+1})^{2} 
\chi_{[ s_{j}, s_{j+1} )}\\
&  \le& 
\mu(w_{_{N}} ^*   \widetilde{a}
w_{_{N}} 	\widetilde{b}^2 
w_{_{N}} ^*   \widetilde{a}
w_{_{N}} ) \\
&=&\mu\Big( (w_{_{N}}^*   \widetilde{a}w_{_{N}} \widetilde{b}) (w_{_{N}} ^* \widetilde{a}w_{_{N}} \widetilde{b})^*\Big)
\\ &  \stackrel{\text{\cite[Prop. 3.2.10]{DPS}}}{=}&
\mu\Big( (w_{_{N}} ^* \widetilde{a}w_{_{N}} \widetilde{b})^*(w_{_{N}}^* \widetilde{a}w_{_{N}} \widetilde{b}) \Big)
\\ &  \stackrel{\text{\cite[Prop. 3.2.8]{DPS}}}{=}&
\mu(w_{_{N}} ^*   \widetilde{a}
w_{_{N}} \widetilde{b})^2 
\\ & <&
\left(\frac{N}{N-1}\right)^4	\sum_{-\infty <j<\infty } (t'_{j+1})^{2} (t_{j+1})^{2} 
\chi_{[ s_{j}, s_{j+1} )},
\end{eqnarray*}
which implies that 
\begin{align} \label{1.5.5}
	\sum_{-\infty <j<\infty } t'_{j+1} t_{j+1}
	\chi_{[ s_{j}, s_{j+1} )}
	  \le 
	\mu(w_{_{N}} ^*   \widetilde{a}w_{_{N}} \widetilde{b})
	 <
	\left(\frac{N}{N-1}\right)^2	\sum_{-\infty <j<\infty } t'_{j+1} t_{j+1}
	\chi_{[ s_{j}, s_{j+1} )}.
\end{align}
Combining inequalities \eqref{1.5.5} and \eqref{1.5}, we have
	\begin{align*} 
		\left(\frac{N-1}{N}\right)^2\mu(a)\mu(b)
\le 
		 \mu\Big(
	w_{_{N}} ^*   \widetilde{a}
	w_{_{N}} \widetilde{b}
		\Big)
\le 
		\left(	\frac{N}{N-1}
		\right)^2 \mu(a)\mu(b)	
		.
	\end{align*}
	
 For an  arbitrary $\varepsilon>0$. Letting $N$ be large enough and setting $w_{_{\varepsilon}}:=w_{_{N}}$, we obtain a partial isometry $w_{_{\varepsilon}}$ such that
	\begin{align}\label{ww*w}
	w_{_{\varepsilon}} w_{_{\varepsilon}}^*=\bigcup_{-\infty<j<\infty} e_j, \quad
	w_{_{\varepsilon}}^* w_{_{\varepsilon}}=e^b(\mu^b_{\infty}, \infty),	
	\end{align}
	and 
	\begin{align}\label{aeabeb}
	(1-\varepsilon)
	\mu(a)\mu(b)
    \leq
	\mu	(
	w_{_{\varepsilon}}^* \widetilde{a}
	w_{_{\varepsilon}}
	\widetilde{b})
    \leq (1+\varepsilon)
	\mu(a)\mu(b).
	\end{align}

{\bf{Step 3}}.		
	Since $w_{_{\varepsilon}}$ is a partial isometry
with initial projection $e^b(\mu^b_{\infty}, \infty)$ and final projection 
$\bigcup\limits_{-\infty<j<\infty} e_j$, it follows from the equality \eqref{ww*w} that
\begin{align*}  	aw_{_{\varepsilon}}b
	=a w_{_{\varepsilon}}
    \widetilde{b}
	=a\Bigg(e^a(\mu^a_{\infty}, \infty)+\Big(\bigcup_{-\infty<j<\infty} e_j-e^a(\mu^a_{\infty}, \infty)\Big)\Bigg)w_{_{\varepsilon}}\widetilde{b},
\end{align*}
and
\[
	   \begin{cases}
	   \displaystyle\bigcup_{-\infty<j<\infty} e_j-e^a(\mu^a_{\infty}, \infty)=0,& t^a =\infty, \\
	   a\Big(\displaystyle\bigcup_{-\infty<j<\infty} e_j-e^a(\mu^a_{\infty}, \infty)\Big)=0, & \mu^a_{\infty}=0.	
	   \end{cases}
	   \]
Hence, we have
\begin{align}\label{1.6}
aw_{_{\varepsilon}}b=a
e^a(\mu^a_{\infty}, \infty)w_{_{\varepsilon}}\widetilde{b}
	=\widetilde{a}w_{_{\varepsilon}}\widetilde{b}. 
    \end{align}
   In particular,  \begin{align} \label{0.5}w_{_{\varepsilon}}^*aw_{_{\varepsilon}}b=w_{_{\varepsilon}}^*\widetilde{a}w_{_{\varepsilon}}\widetilde{b}. 
    \end{align}
Moreover, by \cite[Proposition 3.2.7(vi)]{DPS}, we have
\begin{align}\label{1.7}
    \mu(w_{_{\varepsilon}}^*\widetilde{a} w_{_{\varepsilon}}
	\widetilde{b}) 
	\leq \mu(\widetilde{a} w_{_{\varepsilon}}
	\widetilde{b})
	=\mu(w_{_{\varepsilon}}w_{_{\varepsilon}}^*\widetilde{a} w_{_{\varepsilon}}
	\widetilde{b})
	\leq \mu(w_{_{\varepsilon}}^*\widetilde{a} w_{_{\varepsilon}}
	\widetilde{b}).
\end{align}
Hence,
\begin{align}\label{1.8}
	\mu	(w_{_{\varepsilon}}^* a w_{{\varepsilon}}b)\stackrel{\eqref{0.5}}{=}
    \mu	(w_{_{\varepsilon}}^* \widetilde{a}w_{{\varepsilon}}\widetilde{b})
    \stackrel{\eqref{1.7}}{=}
	\mu(\widetilde{a} w_{_{\varepsilon}}
	\widetilde{b})
	\stackrel{\eqref{1.6}}{=}\mu( aw_{_{\varepsilon}} b
	).
\end{align}
Thus, 
for  arbitrary $ \varepsilon>0 $,  there exists a partial isometry  $w_{_{\varepsilon}} \in \cM$
such that  
		\begin{align*}
		(1-\varepsilon)
		\mu(a)\mu(b)
		\stackrel{\eqref{aeabeb}}{\le} \mu	(
	w_{_{\varepsilon}}^* a
	w_{_{\varepsilon}}
	b )
         \stackrel{\eqref{1.8}}{=}
		\mu(  aw_{_{\varepsilon}}b)          \stackrel{\eqref{aeabeb}}{\leq}
		(1+\varepsilon)
		\mu(a)\mu(b)
		.
	\end{align*}		
 By the  assumption that $\mu(axb)\in E(0, \infty)$ for all $x\in\cM$, we have 
 $\mu(  aw_{_{\varepsilon}}b) \in E(0,\infty)$.

If $\mu(a),\mu(b)$ are continuous, then, by the definition of a symmetric function space,  we have 
$\mu(a)\mu(b)\in E(0, \infty)$
and 
		\begin{align*}
			(1-\varepsilon)
		\norm{\mu(a)\mu(b)	}_{E(0,\infty)}
		\leq
	\norm{ 
	 aw_{_{\varepsilon}}
		b
	}_{E(\cM, \tau )}
	\leq
		(1+\varepsilon)
	\norm{ 	\mu(a)\mu(b)
	}_{E(0,\infty)}
		.
	\end{align*}
Now, consider 
 the general case when  $\mu(a)$ and $\mu(b)$ are not necessarily continuous. 
There exists a commutative atomless von Neumann subalgebra $\cA_1$ of $\cM_{s(a)}$ such that $a\in A_1$ and $S(\cA_1,\tau)$ is (trace-measure preserving) $*$-isomorphic to $S (0,\tau(s(a)))$  (see e.g. \cite{CS,CKS}).
Let $\{I_i\}_{i=1}^\infty  $ be  a sequence of mutually disjoint intervals such that both $\mu(a)$ and $\mu(b)$ are continuous on $I_i$. 
Let $p_{I_i}$ be the projection in $\cA_1$ corresponding to $\chi_{I_i}$. Similarly, we define projections $q_{I_i}$'s for the operator $b$. 
Using the result obtained above, there exists a partial isometry $w_i$ such that $w_i^* w_i=q_{I_i}$ and $w_i w_i^* = p_{I_i}$ and 
	\begin{align*}
		(1-\varepsilon)
		\mu(a)\mu(b)\chi_{I_i}
		 \le 
		\mu(  aw_{i} b)          \stackrel{\eqref{aeabeb}}{\leq}
		(1+\varepsilon)
		\mu(a)\mu(b)\chi_{I_i}
		.
	\end{align*}	
The proof is complete by taking   $w_\varepsilon = \sum_{i=1}^\infty w_i $, where
the series is considered  in the  strong operator 
topology.  This completes the proof. 
\end{proof}

\begin{lemma}\label{2infnity}
	Let $\cM$ be a   factor equipped with a semi-finite  faithful normal trace $\tau$,  
     and let $E(0,\infty )$ be a symmetrically quasi-normed function space.
	Let 
 $0\leq b\in S(\cM,\tau)$ and $0 \leq a \in \{x \in S(\cM,\tau):\mu^x_{\infty} > 0 \text{ and } t^x < \infty \}$. 
	If $\mu(axb)\in E(0, \infty)$ for all $x\in\cM$,
 then
 $\mu(a)\mu(b)\in E(0, \infty)$,
  and
  	for arbitrary $\varepsilon>0$, there exists a partial isometry  $w_{_{\varepsilon}} \in \cM$ such that
  	 \begin{align*}
		(1-\varepsilon)
		\mu(a)\mu(b)
		\leq
		\mu(a w_{_{\varepsilon}}b)
		\leq			(1+\varepsilon) \mu(a)\mu(b).		\end{align*}
Consequently,  
	\begin{align*}
		(1-\varepsilon)
		\norm{
			\mu(a)\mu(b)
		}_{E(0, \infty)}
		\leq	\norm{ 
		a	w_{_{\varepsilon}} b
		}_{E(\cM, \tau )}
		\leq
		(1+\varepsilon)
		\norm{ \mu(a)\mu(b)
		}_{E(0,\infty)}.		
	\end{align*}		
\end{lemma}

\begin{proof}
We  assume that $\cM$ is atomless. The case when  is atomic can be  proved similarly. 	Arguing similarly as the proof of Lemma \ref{compact}, we may assume that $\mu(a) ,\mu(b)$ are continuous on their supports. 
Without loss of generality, we assume that 
$t^a \leq t^b$.

	{\bf{Step 1}}
	By the definition of $t^a$, we have 
	$\tau\left(e^a (\mu_\infty^a ,\infty )\right) =t^a$ (see e.g.\cite[(3.46)]{DPS}).
	Denote $e_1:=e^a (\mu_\infty^a ,\infty )$.
	
	The assumption that $\cM$ is atomless allows us to
select a projection 
  $e_1'$ in $\cM $ such that 
	$$e^{b} \Big(\mu(t^a ; b ), \infty  \Big)
	\le 
	e_1'
	\le e^{b}\Big[\mu(t^a ; b  ), \infty \Big),
	$$
 $\tau(e_1')=t^a $  and $b	e_1'=	e_1'b$ \cite[Lemmas 3.7.7 and 3.7.10]{DPS}.

	We  now consider positive  operators $ae_1$ and $be_1'$. 
	By Lemma \ref{compact}, 
	we obtain  that 
   for arbitrary $\varepsilon>0$, 
	 there exists a partial isometry  $h_{_{\varepsilon}}\in \cM$   such that
	$h_{_{\varepsilon}} h_{_{\varepsilon}}^*=e_1$, 
	 	$h_{_{\varepsilon}}^* h_{_{\varepsilon}}=e_1'$,
	 	and 
		\begin{align}\label{partineq1}
				(1-\varepsilon)
			\mu(ae_1)\mu(	be_1')
			\leq	\mu	(
		 ae_1
			h_{_{\varepsilon}}
			be_1'
			)
	\leq
		(1+\varepsilon)
			\mu(ae_1)\mu(	be_1').
	\end{align}

	{\bf{Step 2}}
	For a fixed $0<\delta<\mu_{\infty}^a$, by the definition of $\mu_{\infty}^a$, we have $\tau(e^{a}
   (\mu_\infty^a-\delta,\infty ))=\infty$,
which implies that there exists 
a projection $e_2:=e^{a}
   (\mu_\infty^a-\delta,\infty )-e_1$ in $\cM$ such that   $\tau(e_2)=\infty$ and
\begin{align}\label{aineq2}
    0<
    (\mu^a_{\infty}-\delta)    e_2
    \leq
    ae_2    
    \leq 
    \mu^a_{\infty}e_2.
\end{align}
    Denote $e_2'={\bf 1}-e_1'$.
We have   
 $\tau(e_2')=\infty$ 
 and $b e_2'=e_2'b$.
 We consider operators $ae_2$ and $be_2'$. 
    Since $\cM$ is a factor, it follows from \cite[Corollary 6.2.6]{KR-II} that  there exists a  partial isometry  $g_{_{\delta}} \in \cM $ such that  $ g_{_{\delta}} ^* e_2 g_{_{\delta}} = e_2'$.     
    In particular, we have 
    \begin{align} \label{ineqa2}
            (\mu^a_{\infty}-\delta) e_2'   &
        \,\,    =g_{_{\delta}} ^*\left(
        (\mu^a_{\infty}-\delta)
        e_2\right) g_{_{\delta}} \\ &\nonumber
            \stackrel{\eqref{aineq2}}{\leq}
            g_{_{\delta}} ^* (ae_2)g_{_{\delta}} 
        \\ &\nonumber   \stackrel{\eqref{aineq2}}{\leq}
            g_{_{\delta}} ^* (\mu^a_{\infty}e_2)g_{_{\delta}} 
        =\mu^a_{\infty} e_2'.
    \end{align} 
    It is readily verified 
    that $g_{_{\delta}} ^* ae_2g_{_{\delta}}$ is a self-adjoint operator and commutes with 
    $e_2'$ and $(b  e_2')^*=e_2'b = be_2'$ (see e.g. Step 3 in the proof of Lemma \ref{compact}). By \cite[Proposition 2.2.24 (iii), (iv)]{DPS}, we have
	\begin{align*}
        (\mu^a_{\infty}-\delta)^2 (b e_2')^2
        \stackrel{  \eqref{ineqa2} }{ \leq } 
        b e_2'( g_{_{\delta}} ^* ae_2g_{_{\delta}})^2 b  e_2'
        \stackrel{  \eqref{ineqa2} }{ \leq }    
        (\mu^a_{\infty})^2 (    b   e_2')^2.    
\end{align*}
     By \cite[Propositions 3.2.8 and 3.2.10(i)]{DPS}, it follows that 
     \begin{align}\label{part2}
        (\mu^a_{\infty}-\delta) 
        \mu( b  e_2')
    \leq    \mu\left(   g_{_{\delta}} ^* ae_2g_{_{\delta}}b   e_2'
        \right)  \leq  \mu^a_{\infty}   
        \mu( b  e_2').  
     \end{align}
    Combining \eqref{part2} with 
    \begin{align*}
    (\mu^a_{\infty}-\delta) 
        \mu( b  e_2')
        = \frac{\mu_{\infty}^a-\delta}{\mu_{\infty}^a} \mu(\mu_\infty^a e_2) \mu(be_2') 
         \stackrel{ \eqref{aineq2} }{ \geq  } 
        \frac{\mu_{\infty}^a-\delta}{\mu_{\infty}^a} \mu ( ae_2 )
        \mu(be_2'),
        \\
        \mu^a_{\infty} \mu( b   e_2')
         = \frac{\mu_{\infty}^a}{\mu_{\infty}^a-\delta}    \mu((\mu_\infty^a -\delta )e_2)
        \mu(be_2')
         \stackrel{ \eqref{aineq2} }{ \leq  } 
        \frac{\mu_{\infty}^a}{\mu_{\infty}^a-\delta}    \mu (  ae_2 )
        \mu(be_2'),
        \end{align*}
    we have
	\begin{align*}
		\frac{\mu_{\infty}^a-\delta}{\mu_{\infty}^a} \mu ( ae_2 )
		\mu(be_2')
		\leq
		\mu( g_{_{\delta}} ^* ae_2g_{_{\delta}}be_2')			
		\leq
        \frac{\mu_{\infty}^a}{\mu_{\infty}^a-\delta}	 \mu (  ae_2 )
		\mu(be_2'). 
	\end{align*}
Since $\delta$ is arbitrarily taken, it follows that for any $\varepsilon>0$, 
there exists a partial isometry $g^\varepsilon=g_{_\delta} $ such that 
$g^\varepsilon (g^\varepsilon)^*=e_2$, 
$(g^\varepsilon)^* g^\varepsilon =e_2'$,
and
\begin{align}\label{part2ineq1}		(1-\varepsilon)
	\mu(ae_2
)	
	\mu(be_2')
	\leq
	\mu( (g^\varepsilon) ^* ae_2 g^\varepsilon 
	b	e_2') 
	\leq
	(1+\varepsilon)	
	\mu\left(ae_2
	\right)	
	\mu(be_2').
\end{align}
As in  Step 3 in the proof of Lemma \ref{compact}, we have 
$	\mu( (g^\varepsilon) ^* ae_2 g^\varepsilon 
	b	e_2')
    =\mu( ae_2 g^\varepsilon 
	b	e_2')	. $

	{\bf{Step 3 }}	
Observe that, 
\begin{align}\label{muequal}	\mu(a)\mu(b)= \mu\Big( \mu(a)\mu(b) \Big)
	=
	\mu\Big( \mu(ae_1)\mu(be_1')\oplus\mu(ae_2)\mu(be_2')\Big)
.
\end{align}	
		Therefore, 
for arbitrary $\varepsilon >0 $,  we obtain a partial isometry  $	w_{_{\varepsilon}}:=h_{_{\varepsilon}}+g^{\varepsilon} \in \cM$ such that  
$w_{_{\varepsilon}} w_{_{\varepsilon}}^*=e_1+e_2$, 
$w_{_{\varepsilon}}^* w_{_{\varepsilon}}= \bf{1} $,
and 
\begin{eqnarray*}
		(1-\varepsilon) 
		\mu(a)\mu(b)
		&
	 \stackrel{\eqref{muequal}}{=}	& 			(1-\varepsilon) 
		\mu\Big( \mu(ae_1)\mu(be_1')
		\oplus
		\mu(ae_2)\mu(be_2')
		\Big)
		\\
		&
          \stackrel{\eqref{partineq1}, \eqref{part2ineq1}}{\leq}&
		\mu\Big( \mu\left(
        ae_1h_{_{\varepsilon}}
		be_1'\right)\oplus\mu\left(
        ae_2g^\varepsilon  be_2'\right)\Big)
		\\
		& =&
        \mu\Big(
        ae_1h_{_{\varepsilon}} be_1'\oplus
        ae_2g^\varepsilon  be_2'
		\Big)
        	\\
		& =& 
        \mu(a w_{_{\varepsilon}}b)
			\\& 
	\stackrel{\eqref{partineq1}, \eqref{part2ineq1}}
		{\leq} 	&   (1+\varepsilon)  \mu\Big(  \mu(ae_1)\mu(	be_1')  \oplus \mu\left(ae_2
	\right)	
	\mu(be_2')  \Big)	\\
        & \stackrel{\tiny \mbox{\eqref{muequal}
        }}{=} &
        (1+\varepsilon)  \mu(a)\mu(b).			
\end{eqnarray*}
Since $
\mu(a w_{_{\varepsilon}}b)
\in E(0,\infty)$, it follows from   the definition  of $E(0,\infty)$  that  
	$\mu(a)\mu(b)\in E(0, \infty)$ and 
\begin{align*}
	(1-\varepsilon)
	\norm{\mu(a)\mu(b)}_{E(0, \infty)}
	\leq	\norm{ 
	a	w_{_{\varepsilon}} b
    }_{E(\cM, \tau )}
    \leq
	(1+\varepsilon)
	\norm{ \mu(a)\mu(b)
	}_{E(0,\infty)}.		
\end{align*}	
\end{proof}

The following lemma extends \cite[Lemma 5.3 and 5.4]{F84}. 
\begin{lemma}\label{EOEM2}
    Let $\cM$ be a  von Neumann algebra   equipped with a semi-finite  faithful normal trace~$\tau$ and  $a, b\in S(\cM,\tau)$.   
	The range of $S_{a,b}$ is contained in  the symmetrically quasi-normed  operator space   $E(\cM, \tau)$ if and only if $\mu(a)\mu(b)\in E(0, \infty)$.  Moreover,   we have 
	\begin{align*}
 		\norm{S_{a,b}}_{\cM\to E(\cM,\tau)} \geq
 		\norm{   \mu(a )\mu(  b) }_{E (0,\infty)   }.
 	\end{align*}
\end{lemma}
\begin{proof}
	Let  $a, b\in S(\cM,\tau)$ with polar decompositions   $a=u|a|$ and $b=v|b|$. 
    We have $uu^*=r(a)$, $b=|b^*|v$,  and $v^*v=s(b)$\cite[p.~18--19]{DPS}.  By \cite[Proposition 3.2.7(vi)]{DPS},   we have
	\begin{align*}
		\mu(axb) = \mu(u|a|x|b^*|v)\leq \mu(|a|x|b^*|) 
		=\mu(u^*axbv^*) \leq \mu(axb),\quad \forall x\in\cM,
	\end{align*}
	which implies $\mu(axb)=\mu(|a|x|b^*|)$. Hence, it suffices to consider the case when $a,b\ge 0$. 

	By Lemma~\ref{E0EM}, it suffices to prove the necessity.
	By Lemmas~\ref{compact} and \ref{2infnity},   for an  arbitrary $ \varepsilon >0 $,  there exists a partial isometry  $w_{_{\varepsilon}} \in \cM$ such that  
    \begin{align*}
	(1-\varepsilon)
	\mu(a)\mu(b)
	\leq
\mu(aw_{_{\varepsilon}}b)
.
\end{align*}  
    Since $\mu(aw_{_{\varepsilon}}b)\in E(0,\infty)$, it follows from the definition of $E(0,\infty)$ that $\mu(a)\mu(b)\in E(0, \infty)$ and
	\begin{align*}
		(1-\varepsilon)
		\norm{ \mu(a)\mu(b) }_{E(0, \infty)}
		\leq
		\norm{ aw_{_{\varepsilon}}b }_{E(\cM, \tau )}
		.
	\end{align*}
    Thus, by the definition of the norm of $S_{a,b}$, we have
	\begin{align*}
		\norm{S_{a,b} }_{\cM\to E(\cM,\tau)}=\sup_{\scriptscriptstyle \norm{x}_{_{\cM}}=1} 
		\norm{axb }_{E(\cM,\tau)}
		\geq
		\norm{  \mu(a)\mu(b) }_{E(0, \infty)  }.
	\end{align*}
    This completes the proof. 
\end{proof}

Before proceeding to the proof of Theorems \ref{main} and \ref{main2}, we recall the following  results. 
	
\begin{theorem}\cite[Theorem 7]{Sukochev16}\label{Thm7inS16}
If $a, b\in S(\cM,\tau)$, then $ab\vartriangleleft\mu(a)\mu(b)$.
\end{theorem}

\begin{theorem} \cite[Theorem 4.2]{DDSZ}\label{Weyllog}
If \( a, b \in L_{\log_+}(\cM, \tau) \), then   \( ab \prec\prec_{\log} \mu(a)\mu(b) \).
\end{theorem}

We are now ready to prove the main results of this paper.
\begin{proof}[Proof of Theorem \ref{main}]
By  Lemma \ref{EOEM2}, it suffices to prove that
\begin{align*}
	\norm{S_{a,b}}_{\cM\to E(\cM,\tau)}\leq
	\norm{  \mu(a)\mu(b) }_{E(0, \infty)  }
\end{align*}
when $a, b\in S(\cM,\tau)$ and $\mu(a)\mu(b)$ belongs to $E(0, \infty)$.
For any  $x\in\cM$ satisfying $\norm{x}_{_\cM}=1$, by Theorem \ref{Thm7inS16} and \cite[Proposition 3.2.7 (iv)]{DPS},   we have
\begin{align*}
	\mu(axb)
	&\vartriangleleft
	\mu(ax)\mu(b)
	\leq
	\mu(a)\mu(b).
\end{align*}
Thus, by the definition of the norm of $S_{a,b}$ and 
the monotonicity of a symmetric norm with respect to the uniform submajorisation  \cite[Corollary 3.4.3]{LSZ}, we have
\begin{align*}
	\norm{S_{a,b} }_{\cM\to E(\cM,\tau)} 
	&=\sup_{\scriptscriptstyle \norm{x}_{_\cM=1}} \norm{axb }_{E(\cM,\tau)}
	\\
	&= \sup_{\scriptscriptstyle \norm{x}_{_\cM}=1} 
	\norm{\mu(axb) }_{E(0,\infty)}
	\leq
	\norm{  \mu(a)\mu(b) }_{E(0, \infty)  }.
	\end{align*}
This completes the proof.
\end{proof}

\begin{proof}[Proof of Theorem \ref{main2}]
A complete proof follows from Theorem~\ref{EOEM2} and   the same argument as that in the proof of Theorem \ref{main} with Theorem~\ref{Thm7inS16} replaced by Theorem~\ref{Weyllog}.
\end{proof}

Recall  that every geometrically stable operator ideal is closed with respect to the logarithmic submajorisation\cite{SZ18}. 
Below, we show that every symmetric   quasi-norm has an equivalent symmetric quasi-norm which is monotone (in particular, closed) with respect to the logarithmic submajorisation, which extends \cite[Proposition 2]{Fack}.

\begin{proposition}\label{log}
	Let $(E(0,\infty),\left\|\cdot\right\|_E)$ be a symmetrically quasi-normed function space. Then there exists a quasi-norm on $E(0,\infty)$ equivalent to $\left\|\cdot\right\|_E$ and monotone (and close) with respect to the logarithmic submajorisation.
\end{proposition}

Before proving Proposition~\ref{log}, we establish the following result, which unifies \cite[Proposition~2]{Fack} and \cite[Proposition~3.2]{Kalton}. 

The   Aoki--Rolewicz theorem  states that if \( 0 < r \leq 1 \) satisfies \( C = 2^{1/r - 1} \), then there exists a constant \( B \) such that for any finite collection \( x_1, \dots, x_n \in X \), the inequality
\[
\left\| \sum_{k=1}^n x_k \right\| \leq B \left( \sum_{k=1}^n \|x_k\|^r \right)^{1/r}
\]
holds \cite{Kalton03}. Moreover, it is possible to replace the original quasi-norm \( \norm{\cdot} \) with an equivalent \( r \)-subadditive quasi-norm, denoted by \( \norm{\cdot}_2 \), satisfying
\begin{align}\label{pnorm}
\|x_1 + x_2\|_2 \leq \left( \|x_1\|_2^r + \|x_2\|_2^r \right)^{1/r}.
\end{align}

A space \( X \) is called \emph{\( r \)-normable} if inequality \eqref{pnorm} is satisfied. If the quasi-norm on \( X \) is itself \( r \)-subadditive, we say that \( X \) is \emph{\( r \)-normed}. For convenience, it is often assumed---unless stated otherwise---that a quasi-Banach space is \( r \)-normed for some \( r > 0 \).

A metric on $X$ can be defined by \( d(x, y) = \|x - y\|_2^r \). The space \( X \) is said to be a \emph{quasi-Banach space} if it is complete under this metric. Note that if \( X \) is \( r \)-normed for some \( r > 0 \), then the quasi-norm is continuous with respect to the metric topology, see also \cite{NP24}. 

Let $E(0,\infty)$ be a quasi-Banach symmetric function space. 
For $T\in E(0,\infty)$, the function $\varLambda(T):(0,\infty)\to \mathbb{R}^{+}$ is defined by 
$$ \varLambda_t(T)=\exp\left(\int_0^t\ln(\mu(s;T))ds\right)\in\mathbb{R}^{+},\quad t\in(0,\infty).$$
The space $E(0,\infty)$ is called \emph{geometrically stable} if for any $T\in E(0,\infty)$, we have  $\varLambda_t(T)^{1/t}\in E(0,\infty)$  (see \cite{Kalton,Fack,SZ18}). 

It is known that every quasi-Banach ideal is geometrically stable \cite[Proposition~3.2]{Kalton}, and every Banach symmetric function space is geometrically stable  \cite[Proposition~2]{Fack}.  
Below, 
we extend these results to the more general framework of quasi-Banach  symmetric function spaces.

\begin{proposition}\label{geometrically}
Every    quasi-Banach  symmetric function space $(E(0,\infty),\left\|\cdot\right\|_E)$ is geometrically stable. 
\end{proposition}
\begin{proof}
By the Aoki--Rolewicz  Theorem \cite[Theorem 1.3]{Kalton3}, we may assume that $\left\|\cdot\right\|_E$ is a symmetric  $r$-norm for some $0<r\leq1$, that is, for all $S,T\in E(0,\infty)$,
\begin{align}\label{rnorm}
\left\|S+T\right\|_E^r\leq \left\|S\right\|_E^r+\left\|T\right\|_E^r.
\end{align}
For $T\in E(0,\infty)$, by the definition of $\sigma_s$, for each $k\in\mathbb{N}^*$,  we have
\begin{align}\label{2knorm} 
\left\|D_{2^k}\mu(T)\right\|_E\leq 2^{k/r}\left\|\mu(T)\right\|_E .\end{align}
Fix $\theta>1/r$. 
Let $f_T=\sum_{k=0}^\infty 2^{-\theta k}D_{2^k}\mu(T)$.
By the completeness of $E(0,\infty )$, 
$f_T$ converges in $E(0,\infty)$.
 Since 
 $$
  \norm{ 2^{-\theta k}D_{2^k}\mu(T)   }_{E}\stackrel{\eqref{2knorm}}{\leq } 
 2^{k/r-\theta k} 
 \norm{ \mu(T)   }_{E},
 $$ 
 it follows that 
\begin{align*}
   \left\|f_T\right\|_E^r
 &\nonumber  \stackrel{\eqref{rnorm}}{\leq }  
   \sum_{k=0}^\infty  
  \norm{ 2^{-\theta k}D_{2^k}\mu(T)   }_{E}^r 
  \leq  \sum_{k=0}^\infty  
  2^{k(1-r\theta)} 
 \norm{ \mu(T)   }_{E}^r\\ &
 \,\,=
 \frac{1}{1-2^{1-r\theta}}\left\|\mu(T)\right\|_E^r, 
\end{align*}
 which implies that 
\begin{equation}\label{f_T}
\left\|f_T\right\|_E\leq \frac{1}{(1-2^{1-r\theta})^{1/r}}\left\|\mu(T)\right\|_E.
\end{equation}
Fix $t>0$. 
For any $s>0$ such that $t2^{-k}\leq s\leq t2^{-k+1}$, 
we have 
\[
\mu(s;T)\leq \mu(t2^{-k};T)\leq 2^{k\theta}f_T(t)
\leq
\left(\frac{2t}{s}\right)^{\theta}f_T(t).
\]
Hence, we have 
\begin{align}\label{varLambda}\begin{split}
\varLambda_t(T)^{1/t}=\exp\left(\frac{1}{t}\int_0^t\ln\mu(s;T)ds\right)& \leq \exp\left(\frac{1}{t}\int_0^t  \ln \left(  \left(\frac{2t}{s}\right)^{\theta}f_T(t) \right)  ds\right ) \\
&=e^{\theta(\ln(2)+1)}f_T(t).
\end{split}
\end{align}
It follows from $f_T\in E(0,\infty)$ and the definition of $E(0,\infty)$ that 
$\varLambda_t(T)^{1/t} \in E(0,\infty)$.
\end{proof}

The following result is  certainly known to experts, and therefore, its proof is omitted.

	\begin{theorem}\label{completion}
	Assume that $I$ is  $(0, 1)$, $(0, \infty)$, or $\mathbb{N}$.	Let $E(I)\subseteq S(I)$ be a symmetrically  $r$-normed    function space,  \( 0 < r \leq 1 \).  Suppose that  $(E(I), \left\|\cdot\right\|_E )$ is not complete. Then,  \( E(I) \) is isomorphic and isometric to a dense linear subspace of an $r$-Banach   space \( (\bar{E}(I), \norm{\cdot}_{\bar{E}} ) \subseteq S(I)\), i.e., there exists a one-to-one correspondence \( x \leftrightarrow \widetilde{x} \) of \( E(I) \) onto a dense linear subspace of \( \bar{E}(I) \) such that
		\[
		\widetilde{x + y} = \widetilde{x} + \widetilde{y}, \quad  \widetilde{\alpha x} = \alpha \widetilde{x}, \quad \|\widetilde{x}\|_{\bar{E}} = \|x\|_{E}. 
		\]
		The space \( \bar{E}(I) \) is uniquely determined up to isometric isomorphisms.
	\end{theorem}

The term \emph{measure preserving} (see \cite{CSZ24}) for a measurable map $\omega$ 
between measure spaces $(\Omega_1, \mathcal{A}_1, m_1)$ and $(\Omega_2, \mathcal{A}_2, m_2)$ means that 
\[
\forall A \in \mathcal{A}_1, \quad \omega(A) \in \mathcal{A}_2 \quad \text{and} \quad m_2(\omega(A)) = m_1(A).
\]

\begin{lemma}\label{complete}
If  $E(0,\infty)\subseteq S(0,\infty)$ is 	a  symmetrically $r$-normed    function space,  \( 0 < r \leq 1 \),  then   	
its completion $({\bar E}(0,\infty),\left\|\cdot\right\|_{\bar E})\subseteq S(0,\infty)$  
 is  a quasi-Banach  symmetric function space.
\end{lemma}

\begin{proof}

(1)	Recall 
the $r$-norm on the  completion $\bar E(0,\infty)$ is defined by \begin{align}\label{defbarE}
	\norm{f}_{\bar{E}} = \mathop{\lim}\limits_{n \to \infty} \norm{f_n}_E, 
\end{align}
 where   \(\{f_n\}_{n\geq 1}\subseteq E(0, \infty)\) is a representative  Cauchy  sequence of $f\in \bar{E}(0, \infty)$. Moreover, 
 \(\{f_n\}_{n\geq 1}\) is a  Cauchy sequence with respect to  the  measure topology as $n\to\infty$ (see e.g. \cite{Sukochev17}, \cite{HGS17} or  \cite[Proposition 4.4.4]{DPS}).
Recall from \cite[Proposition 2.5.7]{DPS} that, for a  sequence $\{f_n\}_{n\geq 1} \subseteq S(0,\infty)$, $f_n\stackrel{\cT_m}{\longrightarrow}f$ if and only if
$\lim\limits_{n\to \infty} \mu( f_n-f
)=0
$.

For arbitrary $h\in E(0, \infty)$, by the symmetricity of  
  $E(0, \infty)$,    we have that 
$|h|\in E(0, \infty)$  and 
\begin{align}\label{|fn|E}
	\norm{|h|}_E=\norm{h}_E=\norm{h}_{\bar{E}}.
\end{align}

For arbitrary $f,g\in S(0,\infty)$, we have 
$
0\leq \Big|   |f|-|g|
\Big| 
\leq |f-g|.  
$ 
It follows from \cite[Proposition 3.2.7]{DPS}
that 
\begin{align*}
	\mu\left( |f|-|g|
\right)= \mu\left(\Big|   |f|-|g|
\Big| \right)
& \leq 
\mu(|f-g|)=\mu(f-g).
\end{align*} 
In particular, if $f-g\in E(0,\infty)$, then we have 
 $|f|-|g|\in E(0,\infty)$ and 
$\norm{|f|-|g|}_E\leq \norm{f-g}_E$. 

(2) We claim  that  for any $ f\in S(0,\infty)$,  
$f\in {\bar E}(0,\infty)$ if and only if 
$\left|f\right| 
\in {\bar E}(0,\infty)$. Moreover, 
we have 
\begin{align}\label{9.2}
\norm{f}_{\bar E}=\norm{\left|f\right|}_{\bar E}. 
\end{align}

Let $f \in {\bar E}(0,\infty)$. By the definition  of $\bar{E}(0,\infty )$, there exists a Cauchy sequence $ \{f_n\}_{n\geq 1} \subseteq E(0,\infty)$ with respect to the  $\norm{\cdot}_E$-norm topology such that
$f_n\stackrel{\cT_m}{\longrightarrow} f$ 
as $n\to\infty$. 
Since 
$
	\mu\left( |f_n|-|f|
\right)\leq 
\mu(f_n-f) 
$ and $\norm{|f_n|-|f_m|}_E\leq \norm{f_n-f_m}_E$,
it follows that $|f_n|\stackrel{\cT_m}{\longrightarrow} |f|$ 
as $n\to\infty$ and 
$\{|f_n|\}_{n\geq 1}\subseteq  E(0,\infty)$ is a Cauchy sequence with respect to the  $\norm{\cdot}_E$-norm topology.
Hence,  $
 |f|  \in {\bar E}(0,\infty)
$ and 
$\norm{ \left| f\right|  }_{\bar{E}}
\stackrel{\eqref{defbarE}}{=}
\lim\limits_{n\to\infty}
\left\||f_n|\right\|_{E}
\stackrel{\eqref{|fn|E}}{=}
\lim\limits_{n\to\infty}
\left\|f_n\right\|_{E}
\stackrel{\eqref{defbarE}}{=}
 \norm{ f}_{\bar{E}}$.

If $|f|\in \bar{E}(0, \infty)$, then,  arguing similarly, we obtain that $f\in \bar{E}(0,\infty)$.  


(3) 
Suppose that  \( f\in {\bar E}(0,\infty) \) and \( g\in S(0,\infty)\)   such that  \( 0\leq |g| \leq \left|f\right| \).
We claim that $g\in \bar{E}(0,\infty)$ and $ \norm{g}_{\bar E}\le \norm{f}_{\bar E}$. 

Let $\{|f_n|\}_{n\geq 1}\subseteq  E(0,\infty)$ be a  Cauchy  sequence such that $|f_n|\stackrel{\cT_m}{\longrightarrow} |f|$ 
as $n\to\infty$ (see Step ~(2)). 
Define, for each $n\geq 1$,  
\begin{align*}
	0\leq g_n:= |f_n| \wedge |g|
	\leq |f_n|.
\end{align*}

Observe that 
$|g_n-g_m|\leq |f_n-f_m|$ and 
$
	| |g|-g_n|
	\leq 	|   f-f_n | 
$.
 We have 
$\{g_n\}_{n\geq 1}\subseteq  E(0,\infty)$ is a Cauchy sequence 
 such that
$g_n\stackrel{	\cT_m}{\longrightarrow} |g| $  as $n\to\infty$. 
Hence, $|g|\in \bar{E}(0,\infty )$ and 
$g_n \to |g|$ as $n\to \infty $ in $\bar{E}(0,\infty )$.  
Moreover, we have 
\begin{align}\label{gnfn}
		\norm{g_n}_{E}
	\leq 
	\norm{|f_n|}_{E}
	\stackrel{\eqref{|fn|E}}{=}
	\norm{f_n}_{E}.
\end{align}
By Step (2), $g \in \bar{E} (0,\infty)$, 
and 
$$\|g\|_{\bar{E}}
\stackrel{\eqref{9.2}}{=}
\norm{ |g|}_{\bar{E}}
\stackrel{\eqref{defbarE}}{=}
\lim\limits_{n\to\infty}
\|g_n\|_{E} 
\stackrel{\eqref{gnfn}}{\leq }
\lim\limits_{n\to\infty}
\|f_n\|_{E} \stackrel{\eqref{defbarE}}{=}
\norm{f}_{\bar{E}}.$$
This proves the claim.

(4)
We claim that for any $ f\in S(0,\infty)$, 
$f\in \bar{E}(0,\infty)$  if and only if $\mu(f)\in\bar{E}(0,\infty)$. Moreover, 
\begin{align}\label{fE=mufE}
	\norm{f}_{\bar{E}}=\norm{\mu(f)}_{\bar{E}}.
\end{align}

By  \cite[Lemma 2.8]{CSZ24}, 
for arbitrary  $0\le f\in S(0,\infty)
$ and  $\varepsilon>0$,   there exist two
measure preserving maps $\sigma_1:    \operatorname{supp}(\mu(f))\to \operatorname{supp}(f)$ and  $\sigma_2:    \operatorname{supp}(\mu(f))\to\operatorname{supp}(f)$ 
such that 
\begin{align}\label{twosides}
|f|\circ\sigma_1 	\le (1+\varepsilon)\mu(f) \quad \text{ and} \quad
	\frac{1}{1+\varepsilon} 
	\mu(f)
	\le 
	|f|\circ\sigma_2  .    
\end{align}

Suppose first that $f\in \bar{E}(0,\infty)$ with $\{f_n\}_{n\ge 1}\in E(0,\infty)$ converging to $f$. By Step (2),  $|f|\in \bar{E}(0,\infty)$ and there exists a Cauchy sequence   $\{|f_n|\}_{n\geq 1}\subseteq  E(0,\infty)$ such that 
$\lim\limits_{n\to \infty} \mu(|f_n|-|f| )=0. 
$
Note that 
$
	\mu\Big(|f_n| \circ  \sigma_2
\Big)
\leq 
\mu(f_n), 
$ 
which implies that 
$\{|f_n|\circ \sigma_2 \}_{n\geq 1}\subseteq E(0,\infty)$ is a Cauchy sequence 
and 
\begin{align}\label{normfn}
	\norm{|f_n| \circ  \sigma_2}_{E}
\leq
 \norm{f_n}_{E}. 
\end{align} 
Moreover, we have  
$\lim\limits_{n\to \infty}\mu\left(|f_n|\circ \sigma_2-|f|\circ \sigma_2\right)=0,$ 
i.e., 
$|f_n|\circ  \sigma_2 \stackrel{\cT_m}{\longrightarrow} |f|\circ  \sigma_2 $ as $n\to\infty$. 
Hence,  
\begin{align}\label{0.1}
	\norm{|f| \circ  \sigma_2}_{\bar{E}}
	\stackrel{\eqref{9.2}}{=} 
\lim\limits_{n\to\infty}
\norm{|f_n| \circ  \sigma_2}_{E}
\stackrel{\eqref{normfn}}{\leq} 
\lim\limits_{n\to\infty}\norm{f_n }_{E}	
\stackrel{\eqref{9.2}}{=} 
\norm{f }_{\bar{E}}.
\end{align}
Combining with \eqref{twosides} and Step (3), 
we obtain that $\mu(f)\in \bar{E}(0,\infty)$ and 
$$	\frac{1}{1+\varepsilon} 
\norm{ \mu(f)}_{\bar{E}}
\leq 
\norm{ |f|\circ\sigma_2  }_{\bar{E}}
\stackrel{\eqref{0.1}}{\leq} 
\norm{f }_{\bar{E}}. $$
Since $\varepsilon$ is arbitrarily taken, 
it follows that $	
\norm{ \mu(f)}_{\bar{E}}
\leq 
 \norm{f }_{\bar{E}}. $

Arguing similarly, we obtain that 
 if  $\mu(f)\in\bar{E}(0,\infty)$, then 
 $f\in\bar{E}(0,\infty)$ and 
  $	 
 \norm{f }_{\bar{E}}\leq \norm{ \mu(f)}_{\bar{E}}. $ 
 Hence, 
  $	 
 \norm{f }_{\bar{E}}=\norm{ \mu(f)}_{\bar{E}}. $

(5) We claim that the quasi-norm of   ${\bar E}(0,\infty)$ is symmetric.

If \( f \in {\bar E}(0,\infty) \) and \( g \in S(0,\infty) \) such that \( \mu(f) = \mu(g) \), it follows from Step ~(4) that 
$ \mu(g)= \mu(f)  \in {\bar E}(0,\infty)$. 
This  implies that 
$g\in {\bar E}(0,\infty) $ (see Step (4))   and \( \|g\|_{\bar{E}}
\stackrel{\eqref{fE=mufE}}{=}
\|\mu(g)\|_{\bar{E}} = \|\mu(f)\|_{\bar{E}}
\stackrel{\eqref{fE=mufE}}{=}
 \|f\|_{\bar{E}} \).  
That is, $\bar{E}(0,\infty)$ is a quasi-Banach  symmetric function space. 

The proof is complete. 
\end{proof}

\begin{proof}[Proof of Proposition~\ref{log}] 
By the Aoki--Rolewicz theorem, it suffices to consider the case when $\norm{\cdot}_E$ is an $r$-norm for some $0<r\le 1$\footnote{Note that for a symmetrically quasi-normed function space $E(0,\infty )$, the equivalent $r$-norm is also symmetric. This can be readily verified by   the construction in the proof of \cite[Theorem~1.2]{Kalton3}.}. 
By Lemma \ref{complete}, it suffices to prove the result under the assumption that $E(0,\infty)$ is an  $r$-Banach symmetric function space.

For any $T\in E(0,\infty)$, by Proposition~\ref{geometrically}, 
we have $\varLambda_t(T)^{1/t}\in E(0,\infty)$. Define a quasi-norm $\left\|\cdot\right\|_{E_1}$ on $E(0,\infty)$ by
\[
\left\|T\right\|_{E_1} := \left\|\varLambda_t(T)^{1/t}\right\|_E.
\]
We now show that $\left\|\cdot\right\|_{E_1}$ is a quasi-norm equivalent to $\left\|\cdot\right\|_E$, which is  monotone with respect to the logarithmic submajorisation.

 By the definition of $E(0,\infty)$, we have $\left\|\mu(T)\right\|_E = \left\|T\right\|_E$. 
Since   $\mu(T)$ is non-increasing, it follows that  
\begin{align*}
	\varLambda_t(T)^{{1/t}}=\exp\left(\frac{1}{t}\int_0^t\ln(\mu(s;T))ds\right)\geq \exp\left(\frac{1}{t}\int_0^t\ln(\mu(t;T))ds\right)=\mu(T), \quad\forall t>0. 
\end{align*}
Hence, we have that
\begin{align}\label{C_0}
	\left\|T\right\|_E \leq \left\|T\right\|_{E_1}.
\end{align}
Combining \eqref{f_T} with \eqref{varLambda}, we obtain
\begin{equation}\label{equivalent}
\left\|T\right\|_E \stackrel{\eqref{C_0}}{\leq} \left\|T\right\|_{E_1} =\left \|\varLambda_t(T)^{{1/t}}\right\|_E \stackrel{\eqref{f_T},\eqref{varLambda}}{\leq} C\left\|T\right\|_E,
\end{equation}
where $C>0$ is a constant depending only on $r$.  
 Since $\left\|\cdot\right\|_E$ is a quasi-norm, it follows that $\left\|\cdot\right\|_{E_1}$ satisfies the quasi-triangle inequality and is non-negative.
By \cite[Proposition~3.2.7]{DPS}, for arbitrary $\alpha \in \mathbb{C}$, we have
\begin{align*}
 \varLambda_t(\alpha T)^{{1/t}} &=\exp\left(\frac{1}{t}\int_0^t\ln(\mu(s;\alpha T))ds\right)
 \\&
 =\exp\left(\frac{1}{t}\int_0^t\ln |\alpha|+\ln(\mu(s;T))ds\right)=|\alpha|\varLambda_t(T)^{{1/t}},
\end{align*}
which implies the absolute homogeneity of $\left\|\cdot\right\|_{E_1}$.
Thus, it follows from inequality ~\eqref{equivalent} that $\left\|\cdot\right\|_{E_1}$ is indeed a quasi-norm equivalent to $\left\|\cdot\right\|_E$. 

Finally, the monotonicity (and the closedness) of $\norm{\cdot}_{E_1}$ with respect to the logarithmic submajorisation follows immediately from the definition of $\varLambda_t(T)^{{1/t}}$.  
\end{proof}

Even though not every symmetrically  quasi-normed function space is monotone with respect to $\prec\prec_{\log}$ (see Example \ref{Lp} below), Proposition \ref{log} shows that it has an equivalent quasi-norm which is monotone with respect to  $\prec\prec_{\log}$. This together with 
  Theorem \ref{main2} yields  the following corollary.

\begin{corollary}
	Suppose that
		$\cM$ is a   factor equipped with a semi-finite  faithful normal  trace~$\tau$. Let $E(0,\infty)$ be a symmetrically quasi-normed function space. 
		For any operators $a,b \in S(\cM,\tau)$,
		the range of $S_{a,b}$ is contained in 
 	    $E(\cM, \tau)$ if and only if 
 	    $\mu(a)\mu(b)\in E(0,\infty)$.
		Moreover, there exists a constant $C$ depending on $E(0,\infty)$ only, such that 
			\begin{align*}
			\norm{  \mu(a)\mu(b) }_{E}
			\leq
			\norm{S_{a,b}}_{\cM\to E(\cM, \tau)}
			\leq
			C\norm{  \mu(a)\mu(b) }_{E  }		.
		\end{align*}		

\end{corollary}

\begin{example}\label{Lorentz}
Let $(L_{p,q}(0,\infty), \left\|\cdot\right\|_{L_{p,q}})$ be the Lorentz function space for $0 < p < q < \infty$, and set $E(\cM,\tau) = L_{p,q}(\cM,\tau)$ as the corresponding noncommutative Lorentz space. For any $a, b \in S(\cM,\tau)$, the range of $S_{a,b}$ is contained in $L_{p,q}(\cM,\tau)$ if and only if $\mu(a)\mu(b) \in L_{p,q}(0,\infty)$. Moreover,
\[
\left\|S_{a,b}\right\|_{\cM \to L_{p,q}(\cM,\tau)} = \left\|\mu(a)\mu(b)\right\|_{L_{p,q}(0,\infty)}.
\]
\end{example}
\begin{proof}
By Theorem~\ref{main2}, it suffices to show that $\left\|\cdot\right\|_{L_{p,q}}$ is monotone with respect to  the logarithmic submajorisation. For $T\in L_{p,q}(0,\infty)$, the quasi-norm is given by (see \cite{Bennett_S})
\[
\|T\|_{L_{p,q}} = \left( \int_{0}^{\infty} \big( s^{1/p} \mu(s;T) \big)^q \,\frac{ds}{s} \right)^{1/q}.
\]
For $S,T\in L_{p,q}(0,\infty)$ with $S \prec\prec_{\log} T$, i.e.
\[
\int_0^t \log \mu(s; S) \, ds \leq \int_0^t \log \mu(s; T) \, ds, \quad \forall t > 0.
\]
Hence, for any $t>0$, we have 
\[\int_0^t  \log \mu(s; S)^q  \, ds= 
\int_0^t q \log \mu(s; S) \, ds \leq \int_0^t q \log \mu(s; T) \, ds =\int_0^t  \log \mu(s; T) ^q \, ds.
\]
Adding  $\int_0^t (\frac{q}{p}-1) \log s \, ds$ to both sides, we obtain 
\[
\int_0^t \log \big( s^{\frac1p-\frac1q } \mu(s; S) \big)^q \,ds 
\leq \int_0^t \log \big( s^{\frac1p-\frac1q } \mu(s; T) \big)^q \, ds,~\forall t>0 .
\]
By the Hardy--Littlewood--Polya inequality (see \cite[Chapter 1, Theorem D.2]{MOA}, \cite[Proposition 2.4]{Huang2} or  \cite[Proposition~1.3]{Hiai}), we have 
\[
\int_0^t \big( s^{\frac1p} \mu(s; S) \big)^q \,\frac{ds}{s} 
\leq \int_0^t \big( s^{\frac1p} \mu(s; T) \big)^q \,\frac{ds}{s}, \quad \forall t > 0.
\]
Letting $t\to\infty$, we obtain $\left\|S\right\|_{L_{p,q}} \leq \left\|T\right\|_{L_{p,q}}$.
\end{proof}

 The following example is in contrast with Example \ref{Lorentz} above, which shows that the natural quasi-norm of weak $L_p$-spaces are not   monotone with respect to $\prec\prec_{\log}$.

 \begin{example}\label{Lp}
    Let $r,p,q\in (0,\infty)$ be such that $\frac{1}{r}=\frac{1}{p}+\frac{1}{q}$. The quasi-norm $\norm{\cdot}_{r,\infty}$ defined by setting  
    $$\norm{z}_{r,\infty}=\sup_{t>0} t^{1/r} \mu(t;z),~z\in L_{r,\infty} (0,\infty )~(\mbox{or}~B(H)),$$
    is not monotone with respect to the logarithmic submajorisation. 
    Indeed, by \cite[Lemma 7]{SZ21}, there exists $x,y\in L_{r,\infty} (0,\infty ) $ such that 
    $\norm{x}_{p,\infty } =\norm{y}_{q,\infty }=1$ such that 
    $$  \norm{xy}_{r,\infty} \ge (1-\epsilon) \frac{(p+q)^{\frac{1}{p} +\frac{1}{q}}}{q^{\frac{1}{p}} p^{\frac{1}{q}}}\norm{x}_{p,\infty} \norm{y}_{q,\infty} = (1-\epsilon) \frac{(p+q)^{\frac{1}{p} +\frac{1}{q}}}{q^{\frac{1}{p}} p^{\frac{1}{q}}} . $$
    Note that $\mu(x)\mu(y)\le t^{-1/p}t^{-1/q} = t^{-1/r} $. We have 
    $\norm{\mu(x)\mu(y)}_{r,\infty }\le 1$, and therefore
    $$\norm{xy}_{r,\infty} \ge (1-\epsilon) \frac{(p+q)^{\frac{1}{p} +\frac{1}{q}}}{q^{\frac{1}{p}} p^{\frac{1}{q}}} \norm{\mu(x)\mu(y)}_{r,\infty}.$$
    Since $\epsilon>0$ is arbitrarily taken and $ \frac{(p+q)^{\frac{1}{p} +\frac{1}{q}}}{q^{\frac{1}{p}} p^{\frac{1}{q}}} >1$, it follows
    that 
    $$\norm{xy}_{r,\infty} >  \norm{\mu(x)\mu(y)}_{r,\infty}.$$
Hence, Theorem \ref{main2} fails for the natural quasi-norms of weak $L_p$-ideals.
In other words,  the quasi-norm $\norm{\cdot}_{r,\infty}$ is not monotone with respect to the logarithmic submajorisation. 

This fact leads to several open problems, e.g., whether the H\"older type inequalities given in \cite[Proposition 5.8, 5.12]{DDSZ} hold without the assumption of monotonicity with respect to $\prec\prec_{\log}$, and 
how to describe   (positive) linear    isometries   on  (commutative and noncommutative) weak-$L_p$ spaces under the natural quasi-norms (see \cite{Huang2,HS24} and references therein for results concerning isometries on symmetrically (quasi-)normed spaces whose (quasi-)norms are monotone with respect to $\prec\prec_{\log}$). 
\end{example} 

\begin{rem} 
A symmetric function space $E(0,\infty)$ is said to be   fully symmetric if for any  $f\in E(0,\infty)$, $g \in S(0,\infty )$, the inequality  $\int_0^t \mu(s; g)ds \le \int_0^t \mu(s; f)ds$, $0<t<\infty $,  implies that $g \in E(0,\infty)$ and $\left\|g\right\|_{E(0,\infty)} \le\left\|f\right\|_{E(0,\infty)}$. 
By the Hardy--Littlewood--Polya inequality (see \cite[Chapter 1, Theorem D.2]{MOA}), every fully symmetrically  normed function space is monotone with respect to $\prec\prec_{\log}$ (see also \cite{Huang2}). 
Most of the known examples of   symmetrically normed function  spaces are fully spaces\cite{KPS,DPS}. 
Moreover, it seems to be difficult to find a symmetric norm on a symmetric function space which is monotone with respect to $\prec\prec_{\log}$ but is not
fully symmetric\cite[p.826]{DDSZ}. 

It was known \cite[p.~2]{Sukochev19} that there exists a Banach symmetric sequence space whose norm is monotone with respect to $\prec\prec_{\log}$ but which does not admit an equivalent fully symmetric norm (see \cite{SSS} for concrete examples).
In particular, this answers a question raised in \cite[p.826]{DDSZ} in the setting of symmetric sequence spaces.
However, it is still unknown whether there exists  a symmetric function space $E(0,\infty)$ whose norm   is monotone with respect to $\prec\prec_{\log}$ but   not   fully symmetric\cite[p.826]{DDSZ}.

\end{rem}

\bibliographystyle{amsalpha}

\end{document}